\documentclass[12pt, a4paper]{article}

\usepackage{amsfonts, amsmath, amssymb, amsthm}
\usepackage{changepage}
\usepackage{fullpage}
\usepackage{hyperref}
\usepackage{mathabx}
\usepackage{mathdots}
\usepackage{parskip}
\usepackage{tikz}
\usepackage{xurl} 
\usepackage{lipsum}
\usepackage{comment}
\usepackage[shortlabels]{enumitem}

\usetikzlibrary{arrows}

\DeclareMathOperator{\GF}{GF}

\theoremstyle{plain}
\newtheorem{theorem}{Theorem}
\newtheorem{corollary}[theorem]{Corollary}
\newtheorem{proposition}[theorem]{Proposition}
\newtheorem{lemma}[theorem]{Lemma}

\theoremstyle{definition}
\newtheorem{method}[theorem]{Method} 
\newtheorem{remark}[theorem]{Remark}
\newtheorem{example}[theorem]{Example}

\begin{document}

\title{On the representation of C-recursive integer sequences by arithmetic terms}

\author{Mihai Prunescu \footnote{Research Center for Logic, Optimization and Security (LOS) (Faculty of Mathematics and Computer Science, University of Bucharest, Academiei 14, Bucharest (RO-010014), Romania), e-mail: {\tt mihai.prunescu@gmail.com}. Simion Stoilow Institute of Mathematics of the Romanian Academy (Research unit 5, P.\ O.\ Box 1-764, Bucharest (RO-014700), Romania), e-mail: {\tt mihai.prunescu@imar.ro}. Institute of Logic and Data Science (Bucharest, Romania).}, Lorenzo Sauras-Altuzarra \footnote{{Vienna University of Technology (Vienna, Austria), e-mail: {\tt lorenzo@logic.at}.}}}
\date{}
\maketitle

\begin{abstract} We show that, if an integer sequence is given by a linear recurrence of constant rational coefficients, then it can be represented by the difference of two arithmetic {terms that} do not contain any irrational constant. We apply our methods to various Lucas sequences (including the classical Fibonacci sequence), to the sequence of solutions of Pell's equation and to some {other C-recursive sequences of order three}. \end{abstract}



Note: this is a pre-print. The final version of the present article will be published in Journal of Difference Equations and Applications.

{\section{Introduction} \label{SectionIntroduction}}

{In this work, we present a new representation technique for integer sequences that are solutions of linear difference equations with constant rational coefficients. After providing some basic notions in Section \ref{SectionPreliminary} and explaining the importance of the arithmetic terms (the expressions that we use as representations) in Section \ref{SectionHypergeometric}, we show a couple of foundational results in Section \ref{SectionExtraction} and explain the representation procedure in Section \ref{SectionCRecursive}. We apply this method to some famous integer sequences in Section \ref{SectionLucas} and Section \ref{SectionOther}, and we conclude the article with a small summary of key ideas in Section \ref{SectionConclusions} and some directions for further research in Section \ref{SectionFutureWork}.}

{\section{Preliminary concepts} \label{SectionPreliminary}}

In this article we refer to any non-negative integer as a \textbf{natural number}, we denote the set of natural numbers by {$ \mathbb{N} $. In} addition, we consider that the definition domain of a sequence is $ \mathbb{N} $.

Given a non-empty set $ X $, a superset $ Y $ of $ X $ and a non-empty set $ F $ of finitary operations on $ Y $, the \textbf{inductive closure} of $ X $ with respect to $ F $ is the minimum set $ C $ such that $ X \subseteq C \subseteq Y $ and, if $ r $ is a positive integer, $ f $ is an $ r $-ary operation in $ F $ and $ \vec{c} \in C^r $, then $ f ( \vec{c} ) \in C $ (cf.\ Enderton \cite[Section 1.4]{Enderton}).

The \textbf{truncated subtraction}, which is denoted by $ \dotdiv $, is the binary operation given by $$x \dotdiv y = \max ( x - y , 0 ) $$ (see {Vereshchagin} \& Shen \cite[p.\ 141]{VereshchaginShen}).

Given {an integer $ r \geq 1 $}, an $ r $-variate \textbf{arithmetic term} in variables $ n_1 $, $ \ldots $, $ n_r $ is an element of the inductive closure of $ \mathbb{N} \cup \{ n_1 , \ldots , n_r \} $ with respect to the binary operations given by \[ x + y , x \dotdiv y , x \cdot y , \left\lfloor x / y \right\rfloor , x^y \] (cf.\ {Prunescu \cite{Prunescu}} and Prunescu \& Sauras-Altuzarra {\cite{PrunescuSaurasAltuzarra, PrunescuSaurasAltuzarra2}}).

Note that, as $$ x \bmod y = x \dotdiv ( y \cdot \left\lfloor x / y \right\rfloor ) , $$ we also use this binary operation.

It is important to remark that we follow the conventions $ 0^0 = 1 $ (see Mendelson \cite[Proposition 3.16]{Mendelson}) and $ \left\lfloor x / 0 \right\rfloor = 0 $ (see {Mazzanti \cite[Subsection 2.1]{Mazzanti}}), so we have $ x \bmod 0 = x $ and $ x \bmod 1 = 0 $.

A univariate \textbf{Kalmar function} is a computable {sequence of natural numbers} whose deterministic computation time is upper-bounded by some sequence of the form $$ 2^{2^{\iddots^{2^n}}} $$ (see Marchenkov \cite[Introduction]{Marchenkov} and Oitavem \cite[Introduction]{Oitavem}).

Most of the usual {sequences of natural numbers} in mathematics are Kalmar functions {(see Mazzanti \cite[Introduction]{Mazzanti})}, and Mazzanti {\cite[Theorem 4.6]{Mazzanti}} proved that every Kalmar function can be represented by an arithmetic term (of the same number of arguments).

Given a ring $ R $, a sequence of terms of $ R $ that satisfies a \textit{homogeneous linear recurrence} of constant \textit{coefficients} in $ R $ (i.e.\ a sequence {$ s ( n ) $ of elements of $ R $} such that, for some {integer $ d \geq 1 $} and $ d $ elements $ \alpha_1 $, $ \ldots $, $ \alpha_d $ of $ R $, we have that $ \alpha_d $ is non-zero and $$ s ( n + d ) + \alpha_1 s ( n + d - 1 ) + \ldots + \alpha_d s ( n ) = 0 $$ for every integer $ n \geq 0 $) is said to be \textbf{C-recursive} of \textbf{order} $ d $ (cf.\ Kauers \& Paule \cite[Section 4.2]{KauersPaule} and Petkovšek \& Zakrajšek \cite[Definition 1]{PetkovsekZakrajsek}).

{And, given} a non-empty subset $ S $ of a ring $ R $, we denote the set of polynomials (resp., rational functions, formal power series) with coefficients in $ S $ and variable $ z $ as $ S [ z ] $ (cf.\ $ S ( z ) $, $ S [[ z ]] $) (cf.\ Hungerford \cite[List of Symbols]{Hungerford}).

{\section{Arithmetic terms vs.\ hypergeometric closed forms} \label{SectionHypergeometric}}

Observe that the total number of operations occurring in an arithmetic term is \textit{fixed} (i.e.\ it does not depend on the variables). Expressions with this property are often called \textit{closed forms} (cf.\ Borwein \& Crandall \cite{BorweinCrandall}).

For example, $ \sum_{k = 0}^n ( 5^k ) $ and $ ( 5^{n + 1} - 1 ) / 4 $ are expressions that represent the same {sequence}, but only the latter is usually considered a closed form because it is the only one for which the total number of operations does not depend on $ n $.

In addition, consider the classical \textbf{Fibonacci sequence}, defined by {the recurrence} $ F ( 0 ) = 0 $, $ F ( 1 ) = 1 $ and $$ F ( n + 2 ) - F ( n + 1 ) - F ( n ) = 0 $$ for every integer $ n \geq 0 $ (see K\v{r}\'{i}\v{z}ek et al.\ \cite[Remark 10.12]{KrizekEtAl}). {It} is known that \begin{equation}\label{SumF} F(n) = \frac{1}{2^{n-1}} \sum _{k=0}^{\left\lfloor (n-1)/2 \right\rfloor} 5^k \binom{n}{2k+1} \end{equation} for every integer $ n \geq 0 $ (see \href{https://oeis.org/A000045}{\texttt{OEIS A000045}}). But the right-hand side of Identity \eqref{SumF} cannot be regarded as a closed form either, because it contains a sum of variable length.

{Now, consider} a field $ K $. The \textbf{consecutive term ratio} of an expression $ e ( n ) $ is the expression $ e ( n + 1 ) / e ( n ) $. A \textbf{hypergeometric term} with respect to $ K $ is a univariate expression whose consecutive term ratio is a rational function on $ K $ (i.e.\ an expression $ e ( n ) $ such that $ e ( n + 1 ) / e ( n ) \in K ( n ) $). And a \textbf{hypergeometric closed form} with respect to $ K $ is a linear combination of hypergeometric terms. See Petkovšek et al.\ \cite[Definition 8.1.1]{PetkovsekEtAl} and Sauras-Altuzarra \cite[Definition 1.4.13]{SaurasAltuzarra}.

For example, $ n ! + 2^n $ is a hypergeometric closed form (indeed, the consecutive term ratios of $ n ! $ and $ 2^n $ are $ n + 1 $ and $ 2 $, respectively), for which we also know an arithmetic-term representation (see Prunescu \& Sauras-Altuzarra \cite{PrunescuSaurasAltuzarra}). However, $ 2^{n^2} $ and $ n^n $ are very common arithmetic terms that are not hypergeometric closed forms.

The problem of how to calculate hypergeometric closed forms of C-recursive sequences (and, in fact, of \textit{holonomic sequences}) is solved thanks to \textit{Petkovšek's complete Hyper algorithm} (see Petkovšek et al.\ \cite[Section 8.9]{PetkovsekEtAl} and Sauras-Altuzarra \cite[Definition 1.4.11 and Section 1.8]{SaurasAltuzarra}).

{An advantage of working with hypergeometric closed forms is that the involved operations have nice cancellation properties, while working with arithmetic terms can be cumbersome because of the presence of the floor function. However, it is often possible to obtain alternative closed forms that avoid the floor function, for example $$ \left\lfloor \frac{n ( n + 1 )}{4} \right\rfloor = \frac{n ( n + 1 ) + \iota^{n ( n + 1 )} - 1}{4} $$ for every integer $ n \geq 0 $ (cf.\ \href{https://oeis.org/A039823}{\texttt{OEIS A039823}}).}

{On the other hand, the} main deficiency of the hypergeometric closed forms is that, from a strict point of view, they are just pseudo-closed forms, because, as they contain sub-expressions that we cannot compute with total precision, they simply encode symbolic operations of variable length.

For example, we have that \begin{equation}\label{HypergeometricClosedFormF} F(n) = \frac{(1+\sqrt{5})^n - (1-\sqrt{5})^n}{2^n \ \sqrt{5}} \end{equation} for every integer $ n \geq 0 $ (see \href{https://oeis.org/A000045}{\texttt{OEIS A000045}}). The right-hand side of Identity \eqref{HypergeometricClosedFormF} is clearly a hypergeometric closed form. And we can see that, as $ \sqrt{5} $ is irrational, one cannot compute exactly neither $(1+\sqrt{5})^n$ nor $(1-\sqrt{5})^n$ when $ n $ is positive. Therefore, $ ( 1 + \sqrt{5} )^n $ is just a notation for the symbolic sum of variable length $$ \sum_{k = 0}^n \binom{n}{k} ( \sqrt{5} )^k . $$

Another important {limitation of the} hypergeometric closed forms is that they essentially rely on holonomicity conditions, and these are rare. As Flajolet et al.\ put it \cite[Note 2]{FlajoletEtAl}, ``\textit{[...] a rough heuristic in this range of problem is the following: almost anything is non-holonomic until it is holonomic by design.}''. In contrast, {as mentioned in Section \ref{SectionPreliminary},} Mazzanti's theorem ensures that we can find arithmetic-term representations for most of the usual integer sequences.

We show, in Theorem \ref{ThmKalmar}, that a C-recursive sequence {$ s ( n ) $} of natural numbers that is defined by a recurrence formula with rational coefficients, is a Kalmar function. It follows, by applying Mazzanti's theorem, that the sequence {$ s ( n ) $} has a representation as a univariate arithmetic term.

The general method of computing arithmetic terms, designed by Mazzanti in order to prove his theorem, produces, typically, extremely long and intricate outputs {(see Prunescu \& Sauras-Altuzarra \cite{PrunescuSaurasAltuzarra2})}. However, for the particular case of C-recursive sequences of natural numbers, we present an ad-hoc idea that allows to generate much shorter arithmetic terms. {In} fact, our {technique (Method \ref{MainMethod})} permits to represent C-recursive integer sequences defined by recurrence laws with rational coefficients as the difference of two arithmetic terms. {It is worth mentioning that a similar technique was introduced by Prunescu \cite{Prunescu}.}

In addition, we provide numerous applications to Lucas sequences, to the sequence of solutions of Pell's equation and to C-recursive {sequences} of order three. {For} the Fibonacci sequence in particular, we find (in Corollary \ref{CorFormulaF}) the arithmetic-term representation \begin{equation}\label{ArithmeticTermF} F(n) = \left \lfloor \frac{3^{n^2 + n}}{3^{2n} \dotdiv ( 3^n + 1 )} \right \rfloor \bmod 3^n , \end{equation} which holds for every integer $ n \geq 0 $. Notice that, unlike in the right-hand side of Identity \eqref{HypergeometricClosedFormF}, in the right-hand side of Identity \eqref{ArithmeticTermF} every sub-term evaluates to an integer, and thus the total number of operations is really fixed.

\section{Term extraction} \label{SectionExtraction}

The \textbf{generating function} in variable $ z $ of a sequence {$ s ( n ) $ of complex numbers}, which is denoted by $ \GF_s ( z ) $, is the formal power series $$ \sum_{i = 0}^\infty s ( i ) z^i $$ (cf.\ Weisstein \cite{Weisstein2}).

Theorem \ref{ThmExtraction1} shows how, under certain conditions, one can extract a term of an integer sequence by means of its generating function.

\begin{theorem}\label{ThmExtraction1} If {$ t ( n ) $ is a sequence of natural numbers}, $ R $ is the radius of convergence of {$ \GF_t ( z ) $} at zero and $ b $, $ m $ and $ n $ are three integers such that $ b \geq 2 $, $ n \geq m \geq 2 $, $ b^{- m} < R $ and $ t ( r ) < b^{r - 2} $ for every integer $ r \geq m $, then \begin{equation}\label{IdentityExtraction} t(n) = \left \lfloor b^{n^2} \GF_t ( b^{- n} ) \right \rfloor \bmod b^n . \end{equation} \end{theorem}

\begin{proof}
 We have that $ b \geq 2 $ and $ n \geq m \geq 2 $, so $ 0 < b^{- n} \leq b^{- m} < R $ and hence we can evaluate {$ \GF_t ( z ) $} at $ b^{-n} $.

By doing so, we obtain $$\GF_t(b^{-n}) = t(0) + \frac{t(1)}{b^n} + \frac{t(2)}{b^{2n}} + \ldots + \frac{t(n)}{b^{n^2}} + \frac{t(n + 1)}{b^{n^2 + n}} + \ldots $$ and thus, for some {integer $ k \geq 0 $}, $$ b^{n^2} \GF_t(b^{-n}) = k b^{n} + t(n) + \frac{t(n + 1)}{b^{n}} + \frac{t(n + 2)}{b^{2 n}} + \frac{t(n + 3)}{b^{3 n}} + \ldots . $$

Let $ v = b^{n^2} \GF_t(b^{-n}) - k b^{n} - t(n) $.

Given {an integer $ i \geq 1 $}, the inequality $$ ni - n - i + 2 \geq i $$ is equivalent with $$ (n - 2)(i - 1) \geq 0 , $$ which holds true because $ n \geq m \geq 2 $.

Therefore every {integer $ i \geq 1 $} satisfies that $$ \frac{t(n + i)}{b^{ni}}< \frac{b^{(n + i) - 2}}{b^{ni}} = \frac{1}{b^{ni- n - i + 2}} \leq \frac{1}{b^i} . $$

It follows that $$ 0 \leq v < \frac{1}{b} + \frac{1}{b^2} + \frac{1}{b^3} + \ldots = \frac{1}{b - 1} \leq 1 $$ and, consequently, $$ \left \lfloor b^{n^2} \GF_t(b^{-n}) \right \rfloor = k b^n + t(n) . $$

Finally, by applying that $ 0 \leq t(n) < b^{n-2} < b^n $, we get $$ t(n) = \left \lfloor b^{n^2} \GF_t(b^{-n}) \right \rfloor \bmod b^n . $$  
\end{proof}

Theorem \ref{ThmExtraction2} shows that, under some stronger conditions than those of Theorem \ref{ThmExtraction1}, Identity \eqref{IdentityExtraction} already holds for every {integer $ n \geq 1 $}.

\begin{theorem}\label{ThmExtraction2} If {$ t ( n ) $ is a sequence of natural numbers}, $ R $ is the radius of convergence of {$ \GF_t ( z ) $} at zero and $ b $ and $ n $ are positive integers such that $ b \geq 8 $, $ b^{- 1} < R $ and $ t ( r ) < b^{r / 3} $ for every {integer $ r \geq 1 $}, then $$ t(n) = \left \lfloor b^{n^2} \GF_t ( b^{- n} ) \right \rfloor \bmod b^n . $$ \end{theorem}

\begin{proof} We have that $ 0 < b^{-n} \leq b^{-1} < R $, so we can evaluate {$ \GF_t ( z ) $} at $ b^{-n} $ and, in an analogous manner than that in the proof of Theorem \ref{ThmExtraction1}, we obtain that, for some {integer $ k \geq 0 $}, $$ b^{n^2} \GF_t(b^{-n}) = k b^{n} + t(n) + \frac{t(n + 1)}{b^{n}} + \frac{t(n + 2)}{b^{2 n}} + \frac{t(n + 3)}{b^{3 n}} + \ldots . $$

Let $ v = b^{n^2} \GF_t(b^{-n}) - k b^{n} - t(n) $.

Given {an integer $ i \geq 1 $}, the inequality $$ ni - \frac{n}{3} - \frac{i}{3} \geq \frac{i}{3} $$ is equivalent with $$ n \geq \frac{2}{3 - 1 / i} , $$ which holds true because $ n $ is positive.

Therefore every {integer $ i \geq 1 $} satisfies that $$ \frac{t(n + i)}{b^{ni}}< \frac{b^{( n + i ) / 3}}{b^{ni}} = \frac{1}{b^{ni - n / 3 - i / 3}} \leq \frac{1}{b^{i / 3}} = \frac{1}{(\sqrt[3]{b})^i} . $$

By applying that $ b \geq 8 $, it follows that $$ 0 \leq v < \frac{1}{\sqrt[3]{b}} + \frac{1}{(\sqrt[3]{b})^2} + \frac{1}{(\sqrt[3]{b})^3} + \ldots = \frac{1}{\sqrt[3]{b} - 1} \leq 1 $$ and, consequently, $$ \left \lfloor b^{n^2} \GF_t(b^{-n}) \right \rfloor = k b^n + t(n) . $$

Finally, by applying that $ 0 \leq t(n) < b^{n/3} < b^n $, we get $$ t(n) = \left \lfloor b^{n^2} \GF_t(b^{-n}) \right \rfloor \bmod b^n . $$ \end{proof}

Example \ref{ExFibonacci} shows that, for well-chosen constants $ b $, the expression $$ \lfloor b^{n^2} ( t ( 0 ) + t ( 1 ) b^{- n} +  t ( 2 ) b^{- 2 n} + \ldots + t ( n ) b^{- n^2} + \ldots ) \rfloor $$ from Identity \eqref{IdentityExtraction} encodes information about the whole tuple $ ( t ( k ) : 0 \leq k \leq n ) $.

\begin{example}\label{ExFibonacci} Consider the Fibonacci sequence {$ F ( n ) $}, $n = 10$ and $b = 10$. As we will see later, $ \GF_F(z) = z / (1 - z - z^2) $, so the number $ \left \lfloor b^{n^2} \GF_t ( b^{- n} ) \right \rfloor $, which can be written as $$ \left \lfloor \frac{b^{n^2 + n}}{b^{2 n} - b^n - 1} \right \rfloor , $$ becomes $$ \textbf{1}000000000\textbf{1}000000000\textbf{2}000000000\textbf{3}000000000\textbf{5}000000000\textbf{8} \ldots $$ $$ \ldots 00000000\textbf{13}00000000\textbf{21}00000000\textbf{34}00000000\textbf{55} . $$ \end{example}

\section{C-recursive sequences} \label{SectionCRecursive}

Proposition \ref{PropCharacterization} is inspired by Wilf's method for solving recurrences (see Wilf \cite[Section 1.2]{Wilf}). It is also the restriction of the easy part of the characterization of the C-recursive sequences over a field $ K $ to the case in which $ K = \mathbb{Q} $.

\begin{proposition}\label{PropCharacterization} (Stanley \cite[Theorem 4.1.1]{Stanley}, Petkovšek \& Zakrajšek \cite[Theorem 1]{PetkovsekZakrajsek}) If {$ s ( n ) $ is a sequence of rational numbers}, $ d $ is a positive integer, $ \alpha_1 $, $ \ldots $, {$ \alpha_d $ are rational numbers}, $ \alpha_d $ is non-zero and $ B ( z ) = 1 + \alpha_1 z + \ldots + \alpha_d z^d \in \mathbb{Q}[z]$, then the following statements are equivalent.
\begin{enumerate}
    \item If $ n $ is a non-negative integer, then $$ s ( n + d ) + \alpha_1 s ( n + d - 1 ) + \ldots + \alpha_d s ( n ) = 0 . $$
    \item There is a polynomial $ A ( z ) \in \mathbb{Q} [ z ] $ such that $ \deg ( A ) < \deg ( B ) $ and $$ \GF_s ( z ) = A ( z ) / B ( z ) \in \mathbb{Q} ( z ) . $$
\end{enumerate}
\end{proposition}

The proof of Proposition \ref{PropCharacterization} is a helpful method to find the generating function of a C-recursive sequence. We sketch it here for the convenience of the reader.

\begin{proof} Let $ A(z) = \sum _{n=0} ^\infty a(n)z^n = \GF_s(z) B(z) $.

Observe that $ A(z) \in \mathbb{Q} [[z]] $, since $\GF_s(z) \in \mathbb{Q}[[z]] $ and $ B(z) \in \mathbb{Q}[z] $.

In addition, $ a(n) = s(n) + \alpha_1 s(n-1) + \ldots + \alpha_d s(n-d) $ for every integer $ n \geq d $.

It is then easy to see that the two statements are equivalent with the condition that $ a(n) = 0 $ for every integer $ n \geq d $. \end{proof}

For Lemma \ref{LemmaGeneralInequality}, which could be considered folklore, we could not find a reference. But, once again, there are reasons to show its proof: it is constructive, and it may help the reader to find the constant $ c $ for a given C-recursive sequence {$ s ( n ) $}.

\begin{lemma}\label{LemmaGeneralInequality} If {$ s ( n ) $ is a C-recursive sequence of complex numbers}, then there is an integer $ c \geq 1 $ such that $ | s(n) | < c^{n+1} $ for every integer $ n \geq 0 $. \end{lemma}

\begin{proof} The hypothesis that {$ s ( n ) $} is C-recursive yields the existence of {an integer $ d \geq 1 $} and $ d $ complex numbers $\alpha_1$, $\ldots$, $\alpha_d $ such that $ s(n+d) =  - \alpha_1 s ( n + d - 1 ) - \ldots - \alpha_d s ( n ) $ for every integer $ n \geq 0 $.

Now, choose a sufficiently large integer $ c \geq 1 $ such that the following inequalities hold: $ d \left (|\alpha_1| + \ldots + |\alpha_d| \right ) < c $, $|s(0)| < c$, $|s(1)| < c^2$, $\ldots$, $|s(d-1)| < c^d$.

We then show by complete induction that $|s(n)| < c^{n+1}$ for every integer $ n \geq 0 $.

Indeed, given an integer $ n \geq d $, suppose that the inequality $|s(k)| < c^{k+1}$ is true for every {$ k \in \{ 0 , \ldots , n - 1 \} $}.

Then $$|s(n)| \leq |\alpha_1| |s(n-1)| + \ldots +|\alpha_d| |s(n-d)| \leq $$ $$ \leq (|\alpha_1| + \ldots + |\alpha_d|) (|s(n-1)| + \ldots + |s(n-d)|) < $$ $$ < ( c / d ) (c^n + \ldots + c^{n-d+1}) \geq ( c / d ) d c^n = c^{n+1} . $$ \end{proof}

\begin{theorem}\label{ThmKalmar} Every C-recursive sequence of natural numbers is a Kalmar function. \end{theorem}

\begin{proof} Let $ d $ be a positive integer, let $ \alpha_1 $, $ \ldots $, $ \alpha_d $ be integers such that $ \alpha_d $ is non-zero, and let {$ s ( n ) $} be the sequence such that $ s ( n + d ) + \alpha_1 s ( n + d - 1 ) + \ldots + \alpha_d s ( n ) = 0 $ for every integer $ n \geq 0 $.

Let $ c $ be a positive integer such that $s(n) < c^{n+1}$ for every integer {$ n \geq 0 $}, which exists because of Lemma \ref{LemmaGeneralInequality}. In fact, recall, from the proof of Lemma \ref{LemmaGeneralInequality}, that $ c $ can be chosen so that the following inequalities hold: $ d \left (|\alpha_1| + \ldots + |\alpha_d| \right ) < c $, $|s(0)| < c$, $|s(1)| < c^2$, $\ldots$, $|s(d-1)| < c^d$.

{Now, fix an integer $ n \geq d $. In order to compute $ s ( n ) $ from its recurrence formula, one has to compute all the previous terms. Concretely}, for any of such terms, one performs $ d $ many multiplications and $ d - 1 $ many additions: indeed, $ s ( k ) = - \alpha_1 s ( k - 1 ) - \ldots - \alpha_d s ( k - d ) $ for every element $ k $ of $ \{ d , \ldots , n \} $.

Therefore, one performs less than $ 2 d n $ many operations in total. And every number $ a $ involved in these operations is such that $ | a | < c^{n + 1} $.

Given any positive integer $ r $, let $ || r || $ denote the length of the bitstring that encodes $ r $.

It is well-known that $ || r || = \left \lfloor \log ( r ) / \log ( 2 ) \right \rfloor + 1 $ for every integer $ r > 0 $, so $$ || c^{n + 1} || = \left \lfloor \dfrac{\log ( c^{n + 1} )}{\log ( 2 )} \right \rfloor + 1 \leq \dfrac{\log ( c^{n + 1} )}{\log ( 2 )} + 1 = \dfrac{\log ( c )}{\log ( 2 )} ( n + 1 ) + 1 < C n $$ for some integer $ C > 0 $. Hence every number $ a $ involved in the aforementioned operations is such that $ || a || \leq C n $.

It is also well-known that, if $ m $, $ u $ and $ v $ are integers such that $ v \neq 0 $ and $ m = \max ( || u || , || v || ) $, then the computation times of $ u + v $ and $ u v $ are upper-bounded by $ m $ and $ m^2 $, respectively. Thus, every aforementioned operation has a computation time of at most $ C^2 n^2 $ and, consequently, the total computation time is then of at most $ 2 d C^2 n^3 $.

Because $ n < 2^{|| n ||} $, we conclude that the total computation time is upper-bounded by $ d C^2 2^{3 || n || + 1} $. \end{proof}

It is then straightforward, from Theorem \ref{ThmKalmar} and Mazzanti's theorem, that every C-recursive sequence of natural numbers is representable by an arithmetic term.

A property in which $ n $ is the only variable, and which fails for a (possibly empty) finite set of values of $ n $ only, is said to hold \textbf{eventually} (cf.\ Weisstein \cite{Weisstein}) or \textbf{almost everywhere} (cf.\ Petkovšek et al.\ \cite[Section 8.2]{PetkovsekEtAl}).

Theorem \ref{ThmMethod} describes a technique to eventually represent a C-recursive integer sequence {$ s ( n ) $} as the difference of two arithmetic terms, and Theorem \ref{ThmExistence} shows that it is applicable as long as {$ s ( n ) $} has some non-zero term and the coefficients of its recurrence formula are rational.

The principle of these representations is the following. For sequences of non-negative integers, the representing arithmetic term is easily deduced from Theorem \ref{ThmExtraction1} or Theorem \ref{ThmExtraction2}. But this cannot be applied for general integer sequences, because of two reasons: first, the proofs of the aforementioned theorems do not work if the sequence contains negative terms; and second, an arithmetic term can take only non-negative values. The solution is to represent the given integer sequence as an algebraic (not truncated) subtraction of two arithmetic terms. 

In Theorem \ref{ThmMethod} we apply a well-known procedure: if $P \in \mathbb Z[x]$ is some polynomial, then $P$ can be written as $ P(x) = P_+(x) - P_-(x) $, where both $P_+(x), P_-(x) \in \mathbb N[x]$. More exactly, in $P_+(x)$ one collects the non-negative coefficients of $P(x)$, while $P_-(x)$ consists of the negative coefficients of $P(x)$, taken with opposite sign.

\begin{theorem}\label{ThmMethod} Consider {a sequence of natural numbers $ t ( n ) $, an integer sequence $ s ( n ) $}, four integers $ b_1 \geq 2 $, $b_2 \geq 8$, $ c \geq 0 $ and $ m \geq 2 $ and four polynomials $ A_+ ( z ) $, $ A_- ( z ) $, $ B_+ ( z ) $ and $ B_- ( z ) $ in $ \mathbb{N} [ z ] $. In addition, let $ R $ be the radius of convergence of {$ \GF_t ( z ) $} at zero, let $ h = \deg ( B_+ - B_- ) $, let $ E ( n, b ) $ be the expression \begin{equation}\label{BigArithmeticTerm} \left \lfloor \frac{b^{n^2 + h n} A_+(b^{- n}) \ \dotdiv \ b^{n^2 + h n} A_-(b^{- n})}{b^{h n} B_+(b^{- n}) \ \dotdiv \ b^{h n} B_-(b^{- n})} \right \rfloor \bmod b^n \end{equation} and suppose that:
\begin{enumerate}
    \item some term of {$ t ( n ) $} is positive,
    \item the equality $ t ( r ) = s ( r ) + c^{r + 1} $ holds for every integer $ r \geq 0 $,
    \item the equality $ \GF_t ( z ) = ( A_+ ( z ) - A_- ( z ) ) / ( B_+ ( z ) - B_- ( z ) ) $ holds,
    \item the fraction $ ( A_+ ( z ) - A_- ( z ) ) / ( B_+ ( z ) - B_- ( z ) ) $ is irreducible,
    \item the inequality $ B_- ( 0 ) < B_+ ( 0 ) $ holds,
    \item\label{Cond6} the inequality $ b_1^{- m} < R $ holds,
    \item\label{Cond7} the inequality $ t ( r ) < b_1^{r - 2} $ holds for every integer $ r \geq m $.
    \end{enumerate}
Then $ s ( r ) = E ( r , b_1 ) - c^{r + 1} $ for every integer $ r \geq m $. \newline If, instead of the conditions \ref{Cond6} and \ref{Cond7}, the sequence {$ t ( n ) $} satisfies the conditions
\begin{enumerate}[(a)]
    \item \label{CondA} the inequality $b_2^{-1} < R$ holds and
    \item \label{CondB} the inequality $t ( r ) < b_2^{r / 3}$ holds for every integer $r \geq 1$,
\end{enumerate}
then $ s ( r ) = E ( r , b_2 ) - c^{r + 1} $ for every integer $ r \geq 1 $. \end{theorem}

\begin{proof} First, notice that both conditions \ref{Cond6} and \ref{CondA} imply that $ R $ is positive.

If we are under the conditions \ref{Cond6} \& \ref{Cond7}, then Theorem \ref{ThmExtraction1} ensures that, for every integer $ r \geq m $, $$ t ( r ) = \left \lfloor b_1^{r^2} \GF_t(b_1^{- r}) \right \rfloor \bmod b_1^r . $$

If we are under the conditions \ref{CondA} \& \ref{CondB} instead, then Theorem \ref{ThmExtraction2} guarantees that, for every integer $r \geq 1$, $$ t ( r ) = \left \lfloor b_2^{r^2} \GF_t(b_2^{- r}) \right \rfloor \bmod b_2^r . $$

Either way, because all terms of {$ t ( n ) $} are non-negative, and at least one of them is positive, we know that $ \GF_t ( z ) $ is positive in $ [ 0 , R ) $.

The fact that the fraction $ ( A_+ ( z ) - A_- ( z ) ) / ( B_+ ( z ) - B_- ( z ) ) $ is irreducible and equal to $ \GF_t ( z ) $ implies, by applying Proposition \ref{PropCharacterization}, that $ \deg ( A_+ ( z ) - A_- ( z ) ) < h $ and the polynomial $ B_+ ( z ) - B_- ( z ) $ has no real root in $ [ 0 , R ) $.

Hence, by applying that $ B_+ ( 0 ) - B_- ( 0 ) > 0 $, we deduce that the polynomial $ B_+ ( z ) - B_- ( z ) $ is positive in $ [ 0 , R ) $ and, because $ \GF_t ( z ) $ is positive in $ [ 0 , R ) $, the polynomial $ A_+ ( z ) - A_- ( z ) $ is also positive in $ [ 0 , R ) $.

Therefore $ A_+(z) > A_-(z) $ and $ B_+(z) > B_-(z) $ for every $ z \in [ 0 , R ) $.

In particular, if $ b $ and $ n $ are two integers such that $ b > 0 $ and $ 0 < b^{- n} < R $, then $ A_+ ( b^{- n} ) > A_- ( b^{- n} ) $ and $ B_+ ( b^{- n} ) > B_- ( b^{- n} ) $, which, together with the fact that $ \deg ( A_+ ( z ) - A_- ( z ) ) < h $, yield that $$ \left \lfloor b^{n^2} \GF_t (b^{- n}) \right \rfloor \bmod b^n = E ( n , b ) . $$

Finally, by applying that $ t ( r ) = s ( r ) + c^{r + 1} $ for every integer $ r \geq 0 $, the conclusions follows. \end{proof}

\begin{remark} The expression \eqref{BigArithmeticTerm} is an arithmetic term. \end{remark}

\begin{remark} From Proposition \ref{PropCharacterization} and Theorem \ref{ThmMethod}, we can infer that the degree of the polynomials in the numerator and denominator of the generating functions are upper-bounded by $ d $. Therefore, the total number of arithmetic operations to perform is linear in $ d $. \end{remark}

Before stating Theorem \ref{ThmExistence}, we prove three easy lemmas which will help us to shorten its proof.

\begin{lemma}\label{LemmaRealInequality} If $ \alpha $, $ \beta $, $ \gamma $, $ \delta $ are real numbers for which $ 1 < \alpha < \beta $, then there is a real number $ \varepsilon $ such that $ {\alpha}^{x + \delta} < {\beta}^{x + \gamma} $ for every real number $ x \geq \varepsilon $. \end{lemma}

\begin{proof}The condition $ {\alpha}^{x + \delta} < {\beta}^{x + \gamma} $ is equivalent with $$ {\alpha}^{\delta} {\beta}^{- \gamma} < {\left ( \frac{\beta}{\alpha} \right )}^x , $$ which certainly holds for every sufficiently large real value of $ x $ because the function $ ( \beta / \alpha )^x $ tends to infinity as $ x $ grows. \end{proof} 

\begin{lemma}\label{Lemman-2Inequality} For every two integers $ b $ and $ c $ such that $ b > c \geq 0 $, there is an integer $ m \geq 3 $ such that $ c^{n + 1} < b^{n - 2} $ for every integer $ n \geq m $. \end{lemma}

\begin{proof} If $ c < 2 $, then take $ m = 3 $. Otherwise, we have that $ 1 < c < b $ so, by applying Lemma \ref{LemmaRealInequality}, there is a real number $ \varepsilon $ such that $ c^{x + 1} < b^{x + ( - 2 )} $ for every real number $ x \geq \varepsilon $. Now, set $ m = \max ( 3 , \lceil \varepsilon \rceil ) $ and the conclusion follows. \end{proof} 

\begin{lemma}\label{Lemman/3Inequality} For every two integers $ b $ and $ c $ such that $ b > c^6 \geq c \geq 0 $, we have that $ c^{n + 1} < b^{n / 3} $ for every integer $ n \geq 1 $. \end{lemma}

\begin{proof} Observe that, for every integer $ n \geq 1 $, the inequality $ n / 6 + 1 / 6 \leq n / 3 $ is equivalent with $ n \geq 1 $, so $$ c^{n + 1} < {\left ( \sqrt[6]{b} \right )}^{n + 1} = b^{n / 6 + 1 / 6} \leq b^{n/3} . $$ \end{proof}

\begin{theorem}\label{ThmExistence} If {$ s ( n ) $ is an integer sequence with} some non-zero term, $ d $ is a positive integer, $ \alpha_1 $, $ \ldots $, {$ \alpha_d $ are rational numbers}, $ \alpha_d $ is non-zero and $$ s ( n + d ) + \alpha_1 s ( n + d - 1 ) + \ldots + \alpha_d s ( n ) = 0 $$ for every integer $ n \geq 0 $, then there are four integers $ b_1 \geq 2 $, $ b_2 \geq 8 $, $ c \geq 0 $ and $ m \geq 3 $ such that one can apply Theorem \ref{ThmMethod} (and, consequently, {$ s ( n ) $} can be written as the difference of two arithmetic terms). \end{theorem}

\begin{proof} We show that we can find $ b_1 \geq 2 $, $ b_2 \geq 8 $, $ c \geq 0 $ and $ m \geq 3$ such that, if {$ t ( n ) := s ( n ) + c^{n + 1} $} and $ R $ is the radius of convergence of {$ \GF_t ( z ) $} at zero, then:
\begin{enumerate}
    \item no term of {$ t ( n ) $} is negative,
    \item some term of {$ t ( n ) $} is positive,
    \item the expression $ \GF_t ( z ) $ belongs to $ \mathbb{Q} ( z ) $,
    \item the inequality $ b_1^{- m} < R $ holds,
    \item the inequality $ t ( n ) < b_1^{n - 2} $ holds for every integer $ n \geq m $,
    \item the inequality $b_2^{-1} < R$ holds and
    \item the inequality $ t ( n ) < b_2^{n / 3} $ holds for every integer $ n \geq 1 $.
\end{enumerate}

If {$ s ( n ) $ is a sequence of natural numbers}, let $ c = 0 $; otherwise let $ c $ be a positive integer such that $ |s(n)| < c^{n+1} $ for every integer $ n \geq 0 $, which exists by applying Lemma \ref{LemmaGeneralInequality} to the fact that {$ s ( n ) $} is C-recursive. In both cases, we get that {$ t ( n ) $ is a sequence of natural numbers}.

If {$ s ( n ) $ is a sequence of natural numbers, then $ s ( n ) = t ( n ) $} (because we took $ c = 0 $ in this case) and consequently some term of {$ t ( n ) $} is positive (because {$ s ( n ) $ has some non-zero term}). Otherwise we also get that some term of {$ t ( n ) $} is positive: indeed, if every term of {$ t ( n ) $} were zero, then $ s ( n ) $ would be equal to {$ - c^{n + 1} $}, in contradiction with the fact that $ | s ( n ) | < c^{n + 1} $ for every integer $ n \geq 0 $.

The coefficients $ \alpha_1 $, $ \ldots $, $ \alpha_d $ are rational so, by applying Proposition \ref{PropCharacterization}, there are two polynomials $ A_s $ and $ B_s $ in $ \mathbb{Q} [ z ] $ such that $ \deg ( A_s ) < \deg ( B_s ) $ and the fraction $ A_s(z) / B_s(z) $, which we can suppose irreducible, is equal to $ \GF_s ( z ) $.

Hence $ \GF_t(z) = A_s(z) / B_s(z) + c / ( 1 - c z ) \in \mathbb Q(z) $ so, by again applying Proposition \ref{PropCharacterization}, the sequence {$ t ( n ) $} is C-recursive.

By again applying Lemma \ref{LemmaGeneralInequality}, there is an integer $ c_t \geq 1 $ such that $t(n) < c_t^{n+1}$ for every integer $ n \geq 0 $.

According to Lemma \ref{Lemman-2Inequality}, there are two integers $ b_1 > c_t \geq 1 $ and $ m \geq 3 $ such that $ t ( n ) < c_t^{n+1} < b_1^{n-2} $ for every integer $ n \geq m $. And, of course, the numbers $ m $ and $ b_1 $ can be chosen large enough to also satisfy the inequality $ b_1^{- m} < R $.

Finally, according to Lemma \ref{Lemman/3Inequality}, there is an integer $ b_2 \geq \max ( 8 , c_t^6 + 1 ) $ such that $ t(n) < c_t^{n+1} < b_2^{n/3} $ for every integer $ n \geq 1 $. And again, the number $ b_2 $ can be chosen large enough to also satisfy the inequality $ b_2^{- 1} < R $. \end{proof}

\begin{remark}\label{RemarkNaturalCase} In the case in which we wish to apply Theorem \ref{ThmMethod} to some C-recursive  sequence {of natural numbers} {$ s ( n ) $}, the proof of Theorem \ref{ThmExistence} shows that it is sufficient to take $ c = 0 $. \end{remark}

{Theorem \ref{ThmExistence}} ensures the existence of an integer $b$ such that the arithmetic term displayed in Theorem \ref{ThmMethod} represents the sequence for every positive argument. But the conditions of {Theorem \ref{ThmExistence}} lead to values of $b$ which are usually larger than necessary. If one wants to check numerically the arithmetic-term representation, smaller values of $ b $ are of interest. For this reason, in order to compute an arithmetic term representation that holds for every positive argument, Method \ref{MainMethod} can be applied instead.

\begin{method}\label{MainMethod} Given a non-zero C-recursive integer sequence {$ s ( n ) $} such that the coefficients of its recurrence formula are rational, this method computes an arithmetic term $ E ( x , y ) $ and two non-negative integers $ b ' $ and $ c $ such that $ s ( n ) = E ( n , b ' ) - c^{n + 1} $ for every integer $ n \geq 1 $.
\begin{enumerate}
    \item If {we have a proof that} no term of {$ s ( n ) $} is negative, then {we can} set $ c = 0 $. Otherwise {we can always} find {an integer $ c \geq 1 $} such that, for every integer $ n \geq 0 $, the inequality $ t ( n ) := s ( n ) + c^{n + 1} > 0 $ holds. In order to find $ c $, one can emulate the proof of Lemma \ref{LemmaGeneralInequality}.
    \item Calculate {$ \GF_t ( z ) $} and its radius of convergence at zero, which we can {call} $ R $. In order to find {$ \GF_t ( z ) $}, compute first {$ \GF_s ( z ) $} with the method shown in the proof of Proposition \ref{PropCharacterization}, and then apply the identity $ \GF_t(z) = \GF_s(z) + c / ( 1 - c z ) $.
    \item Find two integers $ b_1 $ and $ m $ such that $ b_1^{- m} < R $ and $ t ( n ) < b_1^{n - 2} $ for every integer $ n \geq m $. In order to find $b_1$, one can emulate the proof of Lemma \ref{Lemman-2Inequality}.
    \item Write $ \lfloor b^{n^2} \GF_t ( b^{- n} ) \rfloor \bmod b^n $ as an arithmetic term $ E ( n , b ) $ (e.g., as in the statement of Theorem \ref{ThmMethod}).
    \item Find an integer $ b' \geq b_1 $ such that $ s(n) = E ( n , b') - c^{n + 1} $ for the remaining positive integers $ n \in \{ 1 , \ldots , m - 1 \} $. {In order to} find a suitable $b'$, one might first emulate the proof of Lemma \ref{Lemman/3Inequality} to find an integer $ b_2 \geq 8 $ such that $t(n) < b_2^{n/3}$ for every integer $n \geq 1$, and then look, by binary search, for the minimum $ b ' $ in $ \{ b_1 , \ldots , b_2 \} $ such that $ s(n) = E ( n , b') - c^{n + 1} $ for every integer $n\geq 1$.
\end{enumerate} \end{method}

{\begin{remark} If the sequence $ s ( n ) $ has non-negative terms only but, in absence of a proof of this fact, one finds an integer $ c > 0 $ such that $|s(n)| < c^{n+1}$ for every integer $ n \geq 0 $, and considers the sequence $t(n) := s(n) + c^{n+1}$, which has positive terms only, then the procedure works, but the final representation of $s(n)$ will be more complicated than necessary. On the other hand, if the sequence $s(n)$ has some negative term, then it is compulsory to consider the sequence $t(n) := s(n) + c^{n+1}$: indeed, the operation $\bmod$ takes non-negative arguments only, so a sequence with negative terms could not be represented by the formulas given by Theorems \ref{ThmExtraction1} and \ref{ThmExtraction2}. \end{remark}}

All the examples below were constructed by using Method \ref{MainMethod}, which is based on Theorem \ref{ThmExtraction1}. The role of Theorem \ref{ThmExtraction2} is only to guarantee that this strategy is successful.

\section{Lucas sequences}\label{SectionLucas}

Consider two integers $ P $ and $ Q $ for which $ 4 Q \notin \{ P^2 , 0 \} $. The Lucas sequences are defined as follows.

The \textbf{Lucas sequence} of the \textbf{first kind} (resp., \textbf{second kind}) with respect to $ ( P , Q ) $, which is denoted by {$ U ( P , Q ) ( n ) $} (resp., {$ V ( P , Q ) ( n ) $}), is the sequence {$ s ( n ) $} such that $ s ( 0 ) = 0 $ (resp., $ s ( 0 ) = 2 $), $ s ( 1 ) = 1 $ (resp., $ s ( 1 ) = P $) and $ s ( n + 2 ) = P s ( n + 1 ) - Q s ( n ) $ for every integer $ n \geq 0 $ (see the Encyclopedia of Mathematics \cite{Encyclopedia}).

For example, {the sequence $ U ( 1 , - 1 ) ( n ) $ is $ F ( n ) $}, the Fibonacci sequence, which we have already mentioned in {Section \ref{SectionHypergeometric}}.

Theorem \ref{ThmFormulasLucas} provides formulas for $ U ( P , Q ) ( n ) $ and $ V ( P , Q ) ( n ) $, which we denote by $ U (P, Q, n) $ and $ V (P, Q, n) $, respectively.

\begin{theorem}\label{ThmFormulasLucas} (Encyclopedia of Mathematics \cite{Encyclopedia}) If $ n $ is a non-negative integer, $ \gamma = \sqrt{P^2 - 4 Q} $, $ \alpha = ( P + \gamma ) / 2 $ and $ \beta = ( P - \gamma ) / 2 $, then \begin{eqnarray*} U ( P , Q , n ) &=& \dfrac{\alpha^n - \beta^n}{\alpha - \beta} , \\ V ( P , Q , n ) &=& \alpha^n + \beta^n . \end{eqnarray*} \end{theorem} 

Notice that if $ \alpha $ and $ \beta $ from Theorem \ref{ThmFormulasLucas} are rational, then these representations can be easily transformed into arithmetic terms. But the method explained in Method \ref{MainMethod} is applicable even if $ \alpha $ and $ \beta $ are irrational, so in the following lines we explore its application to the case of Lucas sequences of both kinds, in full generality.

Let us denote the generating functions of {the sequences $ U ( P , Q ) ( n ) $} and {$ V ( P , Q ) ( n ) $} by $ u(P, Q, z) $ and $ v(P, Q, z) $, respectively. Corollary \ref{CorGFLucas}, which is straightforward from Proposition \ref{PropCharacterization} and the rules of recurrence, provides formulas for these generating functions.

\begin{corollary}\label{CorGFLucas} (Encyclopedia of Mathematics \cite{Encyclopedia}) We have that \begin{eqnarray*} u(P, Q, z) &=& \frac{z}{1 - Pz + Qz^2} , \\ v(P, Q, z) &=& \frac{2 - Pz}{1 - Pz + Qz^2} . \end{eqnarray*} \end{corollary}

By applying Method \ref{MainMethod}, we obtain Corollary \ref{CorFormulasLucas}.

\begin{corollary}\label{CorFormulasLucas} There is an integer $ c \geq 0 $ such that, for every sufficiently large integer $ b \geq 2 $, the following identities hold for every integer $ n \geq 1 $: \begin{eqnarray*} U ( P , Q , n ) &=& \left\lfloor \frac{cb^{n^2+3n} - (cP-1) b^{n^2 + 2n} + c(Q-1)b^{n^2 + n}}{b^{3n}- (c+P)b^{2n} + (cP + Q) b^n - cQ} \right\rfloor \bmod b^n - c^{n+1} , \\ V ( P , Q , n ) &=& \left\lfloor \frac{(c + 2) b^{n^2+3n} - (2c + P + cP) b^{n^2 + 2n} + c(P+Q) b^{n^2 + n}}{b^{3n}- (c+P)b^{2n} + (cP + Q) b^n - cQ} \right\rfloor \bmod b^n - c^{n+1} . \end{eqnarray*} \end{corollary}

\begin{proof} We do the proof only for {the sequence $ U ( P , Q ) ( n ) $}, since for {$ V ( P , Q ) ( n ) $} it is analogous.

We know that $$ U ( P , Q , n + 2 ) - P U ( P , Q , n + 1 ) + Q U ( P , Q , n ) = 0 $$ for every integer $ n \geq 0 $ and $ U ( P , Q , 1 ) \neq 0 $.

We find an integer $ c \geq 0 $ such that, for every integer $ n \geq 0 $, $ u(P, Q, n) < c^{n+1} $.

Let {$ t ( n ) = U(P,Q,n)+c^{n+1} $}.

Then we have that $$ \GF_t(z) = u(P,Q,z)+ \frac{c}{1 - cz} = \frac{ c-(cP -1)z+c(Q-1)z^2}{1 -(c+P)z+(cP +Q)z^2 - cQz^3}. $$

For this rational generating function, we find a suitable constant $b$. \end{proof}

The right-hand sides of the two identities from the statement of Corollary \ref{CorFormulasLucas} will be denoted by $ U ( b , c , P , Q , n ) $ and $ V ( b , c , P , Q , n ) $, respectively.

In the case that the sequences have only non-negative elements, one can take $ c = 0 $ and the expressions simplify as follows.

\begin{corollary}\label{CorNaturalLucas} If {$ U ( P , Q ) ( n ) $ and $ V ( P , Q ) ( n ) $ are sequences of natural numbers}, then, for every sufficiently large integer $ b \geq 2 $, the following identities hold for every integer $ n \geq 1 $: 
\begin{align*} 
U ( P , Q , n ) & = & U ( b , 0 , P , Q , n ) & = & \left \lfloor \frac{b^{n^2+n}}{b^{2n} - P b^n + Q} \right \rfloor \bmod b^n , \\ 
V ( P , Q , n ) & = & V ( b , 0 , P , Q , n ) & = & \left \lfloor \frac{2 b^{n^2+2n} - P b^{n^2 + n}}{b^{2n} - P b^n + Q} \right \rfloor \bmod b^n . 
\end{align*} 
\end{corollary}

\subsection{The Fibonacci sequence}

In this subsection, we apply the theory from Section \ref{SectionLucas} to the particular case of {$ F ( n ) $}.

First we get Lemma \ref{LemmaInequalityF}, which is a witness of Lemma \ref{LemmaGeneralInequality} for the sequence of Fibonacci.

\begin{lemma}\label{LemmaInequalityF} The inequality $ F(n) < 3^{n-2} $ holds for every integer $ n \geq 3 $. Also, $F(n) < 2^{n-2}$ for every integer $n\geq 4$. \end{lemma}

\begin{proof} The proof goes by induction.

Consider some integer $ b \geq 2 $.

If the inequality $ F(n) < b^{n-2} $ holds for two successive arguments $n$ and $n+1$, then it holds also for $n+2$: indeed, $$ F(n+2) = F(n+1) + F(n) < b^{n-1} + b^{n-2} < 2 b^{n-1} \leq b b^{n-1} = b^n. $$

And, by inspecting the first terms of the sequence {$ F ( n ) $}, we observe that $F(3) < 3$ and $F(4) < 9$, leading to the result for $b = 3$. For $b = 2$ the proof is similar. \end{proof}

The number $ (1 + \sqrt{5})/2 $, denoted by $ \varphi $, is known as the \textbf{golden ratio} (see Guy \cite[Section E25]{Guy}), and Identity \eqref{HypergeometricClosedFormF} can be easily transformed into $$ F(n) = \frac{\varphi^n - ( - \varphi )^{- n}}{2 \varphi - 1} . $$

\begin{corollary}\label{CorFormulaF} If $ n $ is a non-negative integer, then $$ F(n) = \left \lfloor \frac{3^{n^2 + n}}{3^{2n} \dotdiv ( 3^n + 1 )} \right \rfloor \bmod 3^n . $$ \end{corollary}

\begin{proof} {Corollary \ref{CorGFLucas}} yields that $ u(1, -1, z) = z / ( 1 - z - z^2 ) $, from which is easy to see that the radius of convergence at zero of this generating function is $ \varphi - 1 $.

Corollary \ref{CorNaturalLucas} shows that we can write $ U ( b , 0 , 1 , - 1 , n ) $ as in the statement, so it only remains to check whether taking $ b = 3 $ satisfies the conditions of Theorem \ref{ThmMethod} or not. But, certainly, $ 0.03703 \approx 3^{- 3} < \varphi - 1 \approx 0.61803 $ and, by applying Lemma \ref{LemmaInequalityF}, $ F(n) < 3^{n - 2} $ for every integer $ n \geq 3 $.

Finally, a simple computation reveals that the statement also holds if $ 0 \leq n \leq 2 $ (for the case $ n = 0 $, we apply the conventions $ \left\lfloor x / 0 \right\rfloor = 0 $ and $ x \bmod 1 = 0 $ that we mentioned in Section \ref{SectionPreliminary}). \end{proof} 

{Corollary \ref{CorFormulaFExtra}} can be proved in a similar way, by applying Lemma \ref{LemmaInequalityF}.

{\begin{corollary} \label{CorFormulaFExtra} If $ n $ is an integer exceeding one, then $$ F(n) = \left \lfloor \frac{2^{n^2 + n}}{2^{2n} - 2^n - 1} \right \rfloor \bmod 2^n . $$ \end{corollary}}

\subsection{Some {Lucas sequences of non-negative terms}}

In this subsection, we report similar results about some other {Lucas sequences of non-negative terms}, both of the first and of the second kind. Some of these representations are written down explicitly.

\begin{example}\label{ExFormulaLucas} (\href{https://oeis.org/A000032}{\texttt{OEIS A000032}}, \textbf{Lucas numbers})  If $ n $ is a positive integer, then $$ V(1, -1, n) = V(5, 0, 1, -1, n) = \left \lfloor \frac{2 \cdot 5^{n^2 + 2n} - 5^{n^2 + n}}{5^{2n} - 5^n - 1} \right \rfloor \bmod 5^n . $$ \end{example}

From these formulas it is easy to construct arithmetic terms which represent the corresponding sequences (for every non-negative integer). For example, from Example \ref{ExFormulaLucas}, it is straightforward that {the sequence $ V ( 1 , - 1 ) ( n ) $} can be represented by the arithmetic term $$  2( 1 \dotdiv n ) + \left \lfloor \frac{2 \cdot 5^{n^2 + 2n} \dotdiv 5^{n^2 + n}}{5^{2n} \dotdiv ( 5^n + 1 )} \right \rfloor \bmod 5^n . $$ Note that here we applied the conventions $ \left\lfloor x / 0 \right\rfloor = 0 $ and $ x \bmod 1 = 0 $ that we mentioned in Section \ref{SectionPreliminary}. We also apply them in Example \ref{ExFormulaNaturals} and in Example \ref{ExFormulaMersenne}.

\begin{example} (\href{https://oeis.org/A000129}{\texttt{OEIS A000129}}, \textbf{Pell numbers}) If $ n $ is a non-negative integer, then $ U(2, -1, n) = U(3, 0, 2, -1, n) $. \end{example}

\begin{example} (\href{https://oeis.org/A002203}{\texttt{OEIS A002203}}, \textbf{Pell-Lucas numbers}) If $ n $ is a positive integer, then $ V(2, -1, n) = V(9, 0, 2, -1, n) $. \end{example}

\begin{example}\label{ExFormulaNaturals} (\href{https://oeis.org/A001477}{\texttt{OEIS A001477}}, sequence of natural numbers) If $ n $ is a non-negative integer, then $$ U(2, 1, n) = U(4, 0, 2, 1, n) = \left \lfloor \frac{2^{2 n^2 + 2 n}}{2^{4 n} - 2^{2 n + 1} + 1} \right \rfloor \bmod 2^{2 n} = n . $$ \end{example}

\begin{example} \label{ExFormulaTwos} (\href{https://oeis.org/A007395}{\texttt{OEIS A007395}}, all-twos sequence) If $ n $ is a positive integer, then $$ V(2, 1, n) = V(4, 0, 2, 1, n) = \left \lfloor \frac{2^{2 n^2 + 2 n + 1}}{2^{2 n} - 1} \right \rfloor \bmod 2^{2 n} = 2 . $$ \end{example}

{Note that,} in the OEIS, the first argument of the all-twos sequence is set as one instead of as zero.

\begin{example} (\href{https://oeis.org/A001045}{\texttt{OEIS A001045}}, \textbf{Jacobsthal numbers})
 If $ n $ is a non-negative integer, then $ U(1, -2, n) = U(4, 0, 1, -2, n) $. \end{example}

\begin{example} (\href{https://oeis.org/A014551}{\texttt{OEIS A014551}}, \textbf{Jacobsthal-Lucas numbers}) If $ n $ is a positive integer, then $ V(1, -2, n) = V(7, 0, 1, -2, n) $. \end{example}

\begin{example}\label{ExFormulaMersenne} (\href{https://oeis.org/A000225}{\texttt{OEIS A000225}}, \textbf{Mersenne numbers}) If $ n $ is a non-negative integer, then $$ U(3, 2, n) = U(6, 0, 3, 2, n) = \left \lfloor \frac{6^{n^2 + n}}{6^{2n} - 3 \cdot 6^n + 2} \right \rfloor \bmod 6^n = 2^n - 1 . $$ \end{example}

\begin{example} \label{ExFormula2n+1} (\href{https://oeis.org/A000051}{\texttt{OEIS A000051}}, {sequence $ 2^n + 1 $}) If $ n $ is a positive integer, then $$ V(3, 2, n) = V(7, 0, 3, 2, n) = \left \lfloor \frac{2 \cdot 7^{n^2 + 2n} - 3 \cdot 7^{n^2 + n}}{7^{2n} - 3 \cdot 7^n + 2} \right \rfloor \bmod 7^n = 2^n + 1 . $$ \end{example}

\subsection{Other Lucas sequences} 

In this subsection, we apply Method \ref{MainMethod} to obtain formulas for two Lucas sequences that take both positive and negative values.

Once again, the proofs can be recreated by following Method \ref{MainMethod}.

\begin{example} (\href{https://oeis.org/A088137}{\texttt{OEIS A088137}}, \textbf{generalized Gaussian Fibonacci integers}) If $ n $ is a positive integer, then $ U(2,3,n) = U(32, 3, 2, 3, n) = $ $$ \left \lfloor \frac{3 \cdot 32^{n^2 + 3n} - 5 \cdot 32^{n^2 + 2n} + 6 \cdot 32^{n^2 + n}}{32^{3n} - 5\cdot 32^{2n} + 9 \cdot 32^n - 9} \right \rfloor \bmod 32^n - 3^{n+1} . $$ \end{example}

\begin{example} (\href{https://oeis.org/A002249}{\texttt{OEIS A002249}}) If $ n $ is a positive integer, then $$
V(1,2,n) = V(8, 2, 1, 2, n) = \left \lfloor \frac{4 \cdot 8^{n^2 + 3n} - 7 \cdot 8^{n^2 + 2n} + 6 \cdot 8^{n^2 + n}}{8^{3n} - 3\cdot 8^{2n} + 4 \cdot 8^n - 4} \right \rfloor \bmod 8^n - 2^{n+1} . $$ \end{example}

\section{Other C-recursive sequences} \label{SectionOther}

In this section we continue studying important C-recursive sequences that are not Lucas sequences. In the first subsection, we look at the sequences of solutions of a Pell equation. Taken apart, both the sequence of the $x$-values and the sequence of the $y$-values prove to be C-recursive, so they enjoy arithmetic-term representations. And, in the second subsection, we apply Method \ref{MainMethod} to some famous C-recursive {sequences} of order three.

\subsection{Pell's equation}

Consider a non-square integer {$ k \geq 1 $}. The Diophantine equation $ X^2 - k Y^2 = 1 $ is known as \textbf{Pell's equation} (see {Barbeau \cite[Preface]{Barbeau}). Let} $ S $ be the set of solutions $ ( X , Y ) $ in $ \mathbb{N}^2 $ of Pell's equation, which is known to be infinite (see {Grigorieva \cite[Theorem 24]{Grigorieva}). Let $ x ( n ) $ and $ y ( n ) $} be the sequences such that $ \{ ( x ( n ) , y ( n ) ) : n \in \mathbb{N} \} = S $, $ ( x ( 0 ) , y ( 0 ) ) = ( 1 , 0 ) $ and $ x $ is strictly {increasing. The} solution $(x(1), y(1))$ is called {\bf fundamental} because all the other solutions can be computed from it, as Theorem \ref{ThmPellSolutions} shows.

\begin{theorem}\label{ThmPellSolutions} (Barbeau \cite[Section 4.2]{Barbeau}, Rosen \cite[Theorem 13.12]{Rosen}) If $ n $ is a positive integer, then $$ x(n) \pm y(n) \sqrt{k} = {\left ( x(1) \pm y(1) \sqrt{k} \right ) }^n . $$ \end{theorem}

Theorem \ref{ThmPellRecurrence} provides a connection of Pell's equation with the theory of C-recursive sequences. We have seen this result proven for various particular values of the parameter $ k $ (see, for example, K\v{r}\'{i}\v{z}ek et al.\ \cite[Theorem 8.10]{KrizekEtAl2} for $ k = 3 $), but we could not find a reference for the general result. For another interesting connection between Pell's equation and Lucas sequences, see Jones \cite{Jones}.

\begin{theorem}\label{ThmPellRecurrence} If $ n $ is a non-negative integer and {$ s ( n ) \in \{ x ( n ) , y ( n ) \} $}, then $$ s ( n + 2 ) = 2 x ( 1 ) s ( n + 1 ) - s ( n ) . $$ \end{theorem}

\begin{proof} The two roots $ \alpha $ and $ \beta $ of the polynomial $ X^2 - 2 x ( 1 ) X + 1 $ are $ x ( 1 ) + \sqrt{x ( 1 )^2 - 1} $ and $ x ( 1 ) - \sqrt{x ( 1 )^2 - 1} $, respectively.

Because $ ( x ( 1 ) , y ( 1 ) ) $ is a solution of Pell's equation, the numbers $ \alpha $ and $ \beta $ can be written as $ x ( 1 ) + y ( 1 ) \sqrt{k} $ and $ x ( 1 ) - y ( 1 ) \sqrt{k} $, respectively.

And, by applying Theorem \ref{ThmPellSolutions}, we have that $$x(n) = \frac{\alpha^n + \beta^n}{2}, \,\,\,\,
y(n) = \frac{\alpha^n - \beta^n}{2 \sqrt{k}}.$$

Now, note that, if $ \gamma \in \{ \alpha , \beta \} $, then $ \gamma^2 - 2 x ( 1 ) \gamma + 1 = 0 $ and, consequently, $ \gamma^{n + 2} - 2 x ( 1 ) \gamma^{n + 1} +   \gamma^n = 0 $. Therefore, because of the linearity of the corresponding \textit{recurrence operator} (see Sauras-Altuzarra \cite[Section 1.2]{SaurasAltuzarra}), the conclusion follows. \end{proof}

{\begin{corollary} \label{CorPellFormulas} For every sufficiently large integer $ b \geq 2 $, the following identities hold for every integer $n \geq 1$: \begin{eqnarray*} x ( n ) &=& \left \lfloor \frac{b^{n^2 + 2n} - x ( 1 ) b^{n^2 + n}}{b^{2n} - 2 x ( 1 ) b^n + 1} \right \rfloor \bmod b^n , \\ y ( n ) &=& \left \lfloor \frac{y ( 1 ) b^{n^2 + n}}{b^{2n} - 2 x ( 1 ) b^n + 1} \right \rfloor \bmod b^n . \end{eqnarray*} \end{corollary}}

\begin{proof} From the proof of Theorem \ref{ThmPellRecurrence}, it is easy to deduce that \begin{eqnarray*} \GF_x(z) &=& \frac{1 - x ( 1 ) z}{1 - 2 x ( 1 ) z + z^2} \\ \GF_y(z) &=& \frac{y ( 1 ) z}{1 - 2 x ( 1 ) z + z^2} \end{eqnarray*} and, consequently, the conclusion follows. \end{proof}

\begin{example} (\href{https://oeis.org/A001081}{\texttt{OEIS A001081}}, \href{https://oeis.org/A001080}{\texttt{OEIS A001080}}) If $ k = 7 $, then the fundamental solution is $ ( X , Y ) = ( 8 , 3 ) $ so $$ x ( n ) = \left \lfloor \frac{143^{n^2 + 2n} - 8 \cdot 143^{n^2 + n}}{143^{2n} - 16\cdot 143^{n} + 1} \right \rfloor \bmod 143^n $$ for every integer $ n \geq 1 $  and $$ y ( n ) = \left \lfloor \frac{3 \cdot 2^{6 n^2 + 6 n}}{2^{12 n} - 2^{6 n + 4} + 1} \right \rfloor \bmod 2^{6 n} $$ for every integer $ n \geq 0 $. Notice that the last expression was obtained from the formula for $y(n)$ given in {Corollary \ref{CorPellFormulas}} with $b = 64$. \end{example}

\subsection{Some C-recursive sequences of order three}

\begin{example} (\href{https://oeis.org/A000073}{\texttt{OEIS A000073}}, \textbf{Tribonacci numbers}) The sequence is defined by the recurrence $s(0) = s(1) = 0$, $s(2) = 1$ and $$ s(n) = s(n-1) + s(n-2) + s(n-3) $$ for every integer $ n \geq 3 $. {Its} generating function is $$ \frac{z^2}{1 - z - z^2 - z^3} . $$ For every integer $ n \geq 0 $, the $ n $-th Tribonacci number is $$ \left \lfloor \frac{2^{n^2 + n}}{2^{3n} - 2^{2n} - 2^n - 1} \right \rfloor \bmod 2^n . $$ \end{example}

\begin{example} (\href{https://oeis.org/A000931}{\texttt{OEIS A000931}},\textbf{Padovan numbers}) The sequence is defined by the recurrence $s(0) = 1$, $s(1) = s(2) = 0$ and $$ s(n) = s(n-2) + s(n-3) $$ for every integer $ n \geq 3 $. Its generating function is $$ \frac{1-z^2}{1-z^2-z^3} . $$ For every integer $ n \geq 1 $, the $ n $-th Padovan number is $$ \left \lfloor \frac{2^{n^2 +3 n} - 2^{n^2 + n} }{2^{3n} - 2^n - 1} \right \rfloor \bmod 2^n . $$ \end{example}

\begin{example} (\href{https://oeis.org/A000930}{\texttt{OEIS A000930}}, \textbf{Narayana's cows sequence}) The sequence is defined by the recurrence $s(0) = s(1) = s(2) = 1$ and $$ s(n) = s(n-1) + s(n-3) $$ for every integer $ n \geq 3 $. Its generating function is
 $$ \frac{1}{1-z-z^3} .$$
For every integer $ n \geq 1 $, the $ n $-th term of the Narayana's cows sequence is $$ \left \lfloor \frac{2^{n^2 +3 n} }{2^{3n} - 2^{2n} - 1} \right \rfloor \bmod 2^n . $$ \end{example}

\begin{example} (\textbf{Fibonacci convolution sequence}) Given an integer $ r \geq 0 $, the $ r $-th Fibonacci convolution sequence is that whose generating function is $$ \left( \frac{z}{1 - z - z^2} \right)^{r + 1} $$ (see Bicknell-Johnson \& Hoggatt \cite[Section 1]{BicknellJohnsonHoggatt}).

For example, the zeroth Fibonacci convolution sequence is the Fibonacci sequence.

If $ n $ is a non-negative integer and $ ( r , b ) \in \{ ( 1 , 4 ) , ( 2 , 2 ) , ( 3 , 3 ) , ( 4 , 3 ) \} $, then the $ n $-th term of the $ r $-th Fibonacci convolution sequence is $$ \left \lfloor \dfrac{b^{n^2 + r n + n}}{{(b^{2 n} - b^{n} - 1)}^{r + 1}} \right \rfloor \bmod b^n . $$ 

The first Fibonacci convoluted sequence is \href{https://oeis.org/A001629}{\texttt{OEIS A001629}}, while \href{https://oeis.org/A001628}{\texttt{OEIS A001628}}, \href{https://oeis.org/A001872}{\texttt{OEIS A001872}} and \href{https://oeis.org/A001873}{\texttt{OEIS A001873}} are the sequences of positive terms of the second, third and fourth Fibonacci convoluted sequence, respectively.
\end{example}

\section{Conclusions} \label{SectionConclusions}

\begin{enumerate}
    \item Arithmetic terms are closed forms involving integer numbers only: the evaluation of an arithmetic term consists of performing a fixed number (i.e.\ a number that is independent of the variables of the term) of arithmetic operations in a given order (see Section \ref{SectionPreliminary} {and Section \ref{SectionHypergeometric}}).
    \item Given a non-zero C-recursive integer sequence {$ s ( n ) $} such that the coefficients of its recurrence formula are rational, Method \ref{MainMethod} computes two integers $ b \geq 2 $ and $ c \geq 0 $ and an arithmetic term $ E ( x , y ) $ such that \begin{equation} \label{ArithmeticTermRepresentation} s ( n ) = E ( n , b ) - c^{n + 1} \end{equation} for every {integer $ n \geq 1 $}.
    \item If no term of {$ s ( n ) $} is negative, then $ c $ can be set as zero (see Remark \ref{RemarkNaturalCase}).
    \item If $ t ( n ) := s ( n ) + c^{n + 1} $, then {$ \GF_t ( z ) $} is a rational function and the value $ E ( n , b ) $ coincides with $$ \left \lfloor b^{n^2} \GF_t ( b^{- n} ) \right \rfloor \bmod b^n $$ (see the proof of Theorem \ref{ThmExistence}).
\end{enumerate}

{\section{Future work} \label{SectionFutureWork}}

{The arithmetic terms outputted by Method \ref{MainMethod} are not necessarily optimal in length. Examples \ref{ExFormulaNaturals}, \ref{ExFormulaTwos}, \ref{ExFormulaMersenne} and \ref{ExFormula2n+1} already exhibit this deficiency. As another example, consider the sequence $ y ( n ) $ that is defined by the recurrence $ y(1) = 1 $, $ y(2) = 2 $, $ y(3) = 2 $, $ y(4) = 3 $, $ y(5) = 2 $, $ y(6) = 3 $, $ y(7) = 3 $ and $$ y ( n + 7 ) = y ( n + 6 ) + y ( n + 1 ) - y ( n ) $$ for every integer $ n \geq 1 $. Method \ref{MainMethod} yields the identity $$ y ( n ) = \left\lfloor \frac{2^{n^2+5n-6}+2^{n^2+4n-5}+2^{n^2+2n-3}-2^{n^2+n-2}+2^{n^2-1}-2^{n^2-n}}{2^{7n-7}-2^{6n-6}-2^{n-1}+1} \right\rfloor \bmod{2^{n-1}} , $$ which holds for every integer $ n \geq 3 $, but it is known that $$ y ( n ) = \left\lfloor \frac{n}{2} \right\rfloor - \left\lfloor \frac{n + 1}{6} \right\rfloor + 1 $$ for every integer $ n \geq 1 $ (see \href{https://oeis.org/A103469}{\texttt{OEIS A103469}}). Therefore, one future goal is to develop a system that computes the shortest arithmetic-term representation for a given C-recursive sequence.}

{As commented in Section \ref{SectionHypergeometric}, Mazzanti's method of arithmetic-term calculation admits any input from the ``huge'' class of Kalmar functions, but produces outputs that are typically too large for practical use. Another goal is then to extend the ad-hoc techniques for the ``small'' class of C-recursive sequences, such as Method \ref{MainMethod} or the one devised by Prunescu \cite{Prunescu}, to slightly wider classes, such as the class of \textit{holonomic sequences} or even the class of $ {\textrm{C}}^2 $\textit{-finite sequences} (see Jiménez-Pastor et al.\ \cite{JimenezPastorEtAl}).}

{Finally, consider the problem ``\textit{Given a sequence $ y ( n ) $ (in a class $ \mathcal{A} $ of sequences), decide constructively the existence of a sequence $ x ( n ) $ (in a class $ \mathcal{B} $ of sequences) such that $ x ( n + 1 ) - x ( n ) = y ( n ) $.}'', which is known as the \textbf{Telescoping Problem} (cf.\ Sauras-Altuzarra \cite[Section 5.3]{SaurasAltuzarra} and Schneider \cite[Section 2]{Schneider}). The Telescoping Problem is a summation problem, because its resolution yields an identity of the form $$ \sum_{k = a}^b y ( k ) = x ( b + 1 ) - x ( a ) , $$ for some integers $ a $ and $ b $ such that $ 0 \leq a \leq b $. As the present paper revolves around differences of arithmetic terms, the following instance of the Telescoping Problem might be another natural question: ``\textit{Given an arithmetic term $ y ( n ) $, decide constructively the existence of an arithmetic term $ x ( n ) $ such that $ x ( n + 1 ) - x ( n ) = y ( n ) $.}''.}

\section{Acknowledgments}

The authors thank the anonymous referees for their attentive review and their very interesting comments, which helped, in particular, to greatly improve the presentation of the theory.

The second author was partially supported by FWF Austria (project number P 36571-N). Both authors were partially supported by Bitdefender (Research in Pairs in Bucharest Program).

\end{document}